\documentclass[12pt]{article}
\usepackage{amssymb}
\usepackage{amsmath,bm}
\usepackage{graphics,mathrsfs}
\usepackage{graphicx,epsfig}
\usepackage{subfigure}
\usepackage{setspace}
\usepackage{enumerate}
\usepackage{enumitem}
\usepackage{varioref}
\usepackage{natbib}
\usepackage{color}
\setcitestyle{authoryear,open={(},close={)}}
\usepackage{amsthm}
\makeatletter \oddsidemargin  -.1in \evensidemargin -.1in
\begin{document}
\newcommand{\bea}{\begin{eqnarray}}
\newcommand{\eea}{\end{eqnarray}}
\newcommand{\nn}{\nonumber}
\newcommand{\bee}{\begin{eqnarray*}}
\newcommand{\eee}{\end{eqnarray*}}
\newcommand{\lb}{\label}
\newcommand{\nii}{\noindent}
\newcommand{\ii}{\indent}
\newtheorem{theorem}{Theorem}[section]
\newtheorem{example}{Example}[section]
\newtheorem{counterexample}{Counterexample}[section]
\newtheorem{corollary}{Corollary}[section]
\newtheorem{definition}{Definition}[section]
\newtheorem{lemma}{Lemma}[section]
\newtheorem{remark}{Remark}[section]
\newtheorem{proposition}{Proposition}[section]
\renewcommand{\theequation}{\thesection.\arabic{equation}}
\renewcommand{\labelenumi}{(\roman{enumi})}
\title{\bf Weighted fractional generalized cumulative past entropy and its properties}
\author{Suchandan {\bf Kayal}$^{1}$\thanks{Corresponding author : Suchandan Kayal (kayals@nitrkl.ac.in,~suchandan.kayal@gmail.com)}~ and N. {\bf Balakrishnan}$^{2}$\thanks{bala@mcmaster.ca}}

\maketitle
\maketitle \noindent {\it $^{1}$Department of Mathematics, National Institute of
	Technology Rourkela, Rourkela-769008, Odisha, India} \\
{\it $^{2}$Department of Mathematics and Statistics,
	McMaster University, Hamilton, Ontario L8S 4K1,
	Canada}
\date{}
\maketitle
\begin{center}
{\large \bf Abstract}
\end{center}
In this paper, we introduce weighted fractional generalized cumulative past entropy of a nonnegative absolutely continuous random variable with bounded support. Various properties of the proposed weighted fractional measure are studied. Bounds and stochastic orderings are derived. A connection between the proposed measure and the left-sided Riemann-Liouville fractional integral is established. Further, the proposed measure is studied for the proportional reversed hazard rate models. Next, a nonparametric estimator of the weighted fractional generalized cumulative past entropy is suggested based on empirical distribution function. Various examples with a real life data set are considered for the illustration purposes. Finally, large sample properties of the proposed empirical estimator are studied.
\\
\\
{\large \bf Keywords:} Weighted generalized cumulative past entropy; Fractional calculus; Stochastic ordering; Reversed hazard rate model; Empirical cumulative distribution function; Central limit theorem.
\\
\\
{\large \bf 2020 Mathematics Subject Classifications:}  $94A17$, $60E15$, $26A33$
\section{Introduction}
Entropy plays an important role in several areas of statistical mechanics and information theory. In statistical mechanics, the most widely applied form of entropy was proposed by Boltzmann and Gibbs, and in information theory, that was introduced by Shannon. Due to the growing applicability of the entropy measures, various generalizations were proposed and their information theoretic properties
were studied. See, for instance, \cite{renyi1961measures} and \cite{tsallis1988possible}. We recall that most of the generalized entropies were developed based on the concept of deformed logarithm. But, two generalized concept of entropies: fractal (see \cite{wang2003extensive}) and fractional (\cite{ubriaco2009entropies}) entropies were proposed based on the natural logarithm.  Let $P=(p_1,\ldots,p_n)$ be the probability mass function of a discrete random variable $X.$ The  Boltzmann-Gibbs-Shannon entropy of $X$ can be defined through an equation involving the ordinary derivative as
\begin{eqnarray}\label{eq1.1}
H_{BGS}(X)=-\frac{d}{du}\sum_{i=1}^{n}p_{i}^{u}\Big|_{u=1}.
\end{eqnarray}
\cite{ubriaco2009entropies} proposed a new entropy measure known as the fractional entropy after replacing the ordinary derivative by the Weyl fractional derivative (see \cite{ferrari2018weyl}) in (\ref{eq1.1}).  It is given by
\begin{eqnarray}\label{eq1.2}
H_{\alpha}(X)=\sum_{i=1}^{n}p_{i}(-\ln p_{i})^{\alpha},~~0\le \alpha\le1.
\end{eqnarray}
The fractional entropy in (\ref{eq1.2}) is positive, concave and non-additive. Further, one can recover the Shannon entropy (see \cite{shannon1948note}) from (\ref{eq1.2}) under $\alpha=1.$ From (\ref{eq1.2}), we notice that the measure of information is mainly a function of probabilities of occurrence of various events. However, we often face with many situations (see \cite{guiacsu1971weighted}) in different fields, where the probabilities and qualitative characteristics of events need to be taken into account for better uncertainty analysis. As a result, the concept of weighted entropy was introduced by \cite{guiacsu1971weighted}, which is given by
\begin{eqnarray}\label{eq1.3}
H^{w}(X)=-\sum_{i=1}^{n}w_{i}p_{i}\ln p_{i},
\end{eqnarray}
where $w_{i}$ is a nonnegative number  (known as weight) directly proportional to the importance of the $i$th elementary event. Note that the weights $w_{i}$'s can be equal. Following the same line as in (\ref{eq1.3}), the weighted fractional entropy can be defined as
\begin{eqnarray}\label{eq1.4}
H_{\alpha}^{w}(X)=\sum_{i=1}^{n}w_{i}p_{i}(-\ln p_{i})^{\alpha},~~0\le \alpha\le1.
\end{eqnarray}
Note that for $w_{i}=1,~i=1,\ldots,n$, (\ref{eq1.4}) reduces to the fractional entropy given by (\ref{eq1.2}). Further, (\ref{eq1.4}) equals to the weighted entropy given by (\ref{eq1.3}) when $\alpha=1.$

Recently, motivated by the aspects of the cumulative residual entropy due to \cite{rao2004cumulative} and the fractional entropy given by (\ref{eq1.2}), \cite{xiong2019fractional} introduced a new information measure, known as the fractional cumulative residual entropy. The concept of multiscale fractional cumulative residual entropy was described by \cite{dong2020multiscale}. Very recently, inspired by the cumulative entropy (see \cite{di2009cumulative}) and (\ref{eq1.2}), \cite{di2021fractional} proposed fractional generalized cumulative entropy of a random variable $X$ with bounded support $(0,s)$, which is given by
\begin{eqnarray}\label{eq1.5}
CPE_{\gamma}(X)=\frac{1}{\Gamma(\gamma+1)}\int_{0}^{s} K(x)[-\ln K(x)]^{\gamma}dx,~~\gamma>0,
\end{eqnarray}
where $K$ is the cumulative distribution function (CDF) of  $X.$ The fractional generalized cumulative entropy is a generalization of the cumulative entropy and generalized cumulative entropy proposed by \cite{di2009cumulative} and \cite{kayal2016generalized}, respectively. We remark that the cumulative entropy and generalized cumulative entropy are independent of the location. This property appears as a drawback when quantifying information of an electronics device or a neuron in different intervals having equal widths. Thus, to cope with these situations, various authors proposed length-biased (weighted) information measures. The weighted measures are also called shift-dependent measures by some researchers.  Readers may refer to \cite{di2007weighted}, \cite{misagh2011weighted}, \cite{misagh2016shift}, \cite{das2017weighted}, \cite{kayal2017weighteda}, \cite{kayal2017weighted}, \cite{mirali2017weighted}, \cite{nourbakhsh2017weighted}, \cite{kayal2018weighted} and \cite{kayal2019shift} for some weighted versions of various information measures. The existing weighted information measures and the fractional generalized cumulative entropy in (\ref{eq1.5}) inspire us to consider the weighted fractional generalized cumulative past entropy (WFGCPE), which has been studied in the subsequent sections of this paper. The following definitions will be useful in order to obtain some ordering results for the WFGCPE.
\begin{definition}
	Let $X_{1}$ and $X_{2}$ be two nonnegative absolutely continuous random variables with probability density functions (PDFs) $k_{1},~k_{2}$ and CDFs $K_{1},~K_{2}$, respectively. Then, $X_{1}$ is said to be smaller than $X_{2}$ in the sense of the
	\begin{itemize}
		\item[(i)] usual stochastic order, denoted by $X_{1}\le_{st}X_{2}$, if $K_{2}(x)\le K_{1}(x)$, for all $x\in \mathbb{R}$;
		\item[(ii)] hazard rate order, denoted by $X_{1}\le_{hr} X_{2}$, if $\bar{K}_{2}(x)/\bar{K}_{1}(x)$ is nondecreasing in $x>0,$ where $\bar{K}_{1}=1-K_{1}$ and $\bar{K}_{2}=1-K_{2};$
		\item[(ii)] dispersive order, denoted by $X_{1}\le_{disp}X_{2}$, if  $K_{1}^{-1}(u)-K_{1}^{-1}(v)\ge K_{2}^{-1}(u)-K_{2}^{-1}(v)$, for all $0<u<v<1$, where $K_{1}^{-1}$ and $K_{2}^{-1}$ are the right continuous inverses of $K_1$ and $K_2$, respectively;
		\item[(iii)] decreasing convex order, denoted by $X_{1}\le_{dcx}X_{2},$ if and only if $E(\tau(X_{1}))\le E(\tau(X_{2}))$ holds for all nonincreasing convex real valued functions for which the expectations are defined.		 
	\end{itemize}
\end{definition}

The rest of the paper is organized as follows. In Section $2$, we introduce WFGCPE and study its various properties. Some ordering results are obtained. It is shown that a less dispersed distribution produces smaller uncertainty in terms of the WFGCPE. Some bounds are obtained. Further, a connection of the proposed measure with the fractional calculus is discovered. The proportional reversed hazard model is considered and the WFGCPE is studied under this set up. Section $3$ deals with the estimation of the introduced measure. An empirical WFGCPE estimator is  proposed based on the empirical distribution function. Further, large sample properties of the proposed estimator have been studied. Finally, Section $4$ concludes the paper with some discussions.

Throughout the paper, the random variables are considered as nonnegative random variables. The terms increasing and decreasing are used in wide sense. The differentiation and integration exist whenever they are used. The notation $\mathbb{N}$ denotes the set of natural numbers. Further, throughout the paper, a standard argument $0=0.\ln 0=0.\ln \infty$ is adopted. The prime $\prime$ denotes the first order ordinary derivatve.

\section{Weighted fractional generalized cumulative past entropy\setcounter{equation}{0}}
In this section, we propose WFGCPE and study its various properties. Consider a nonnegative absolutely continuous random variable $X$ with support $(0,s)$ and CDF $K$ and PDF $k$. Then, the WFGCPE of $X$ with a general nonnegative weight function $\psi(x)~(\ge0)$ is defined as
\begin{eqnarray}\label{eq2.1}
CPE_{\gamma}^{\psi}(X)=\frac{1}{\Gamma(\gamma+1)}\int_{0}^{s}\psi(x)K(x)[-\ln K(x)]^{\gamma}dx,~~\gamma>0,
\end{eqnarray}
provided the right-hand-side integral is finite, where $\Gamma$ is a gamma function. From (\ref{eq2.1}), one can easily notice that the information measure $CPE_{\gamma}^{\psi}(X)$ is always nonnegative. It is equal to zero when $X$ is degenerate. Note that the WFGCPE is nonadditive. We recall that an information measure $H$ is additive if
\begin{eqnarray}\label{eq2.2}
H(A+B)=H(A)+H(B),
\end{eqnarray}
for any two probabilitically independent systems $A$ and $B$. If (\ref{eq2.2}) is not satisfied, then the information measure is said to be nonadditive. Several information measures have been proposed in the literature since the introduction of the Shannon entropy. Among those, probably Shannon's entropy and Renyi's entropy (see \cite{renyi1961measures}) are additive and all other generalizations (see, for example, \cite{tsallis1988possible}) are nonadditive. For $\gamma\in \mathbb{N}$ and $s\rightarrow+\infty$, $CPE_{\gamma}^{\psi}(X)$ reduces to the weighted generalized cumulative entropy proposed by \cite{tahmasebi2020extension}. Further, when $\psi(x)=1$, we get the fractional generalized cumulative entropy due to \cite{di2021fractional}. Let $\Psi'(x)=\frac{d}{dx}\Psi(x)=\psi(x).$ Then, when $\gamma\rightarrow 0^{+}$  and $ 0<s<+\infty$, we have from (\ref{eq2.1}),
\begin{eqnarray}
CPE_{\gamma}^{\psi}(X)
&=&\int_{0}^{s}\psi(x)dx-\int_{0}^{s}\psi(x)\bar{K}(x)dx\nonumber\\
&=&\Psi(s)-\Psi(0)-\int_{0}^{s}\psi(x)\left(\int_{x}^{s}k(y)dy\right)dx\nonumber\\
&=& \Psi(s)-\Psi(0)-\int_{0}^{s}\int_{0}^{y}\psi(x)k(y)dx dy\nonumber\\
&=& \Psi(s)-E[\Psi(X)].
\end{eqnarray}
Thus,
\begin{eqnarray}
CPE_{\gamma}^{\psi}(X)
=\left\{
\begin{array}{ll}
\Psi(s)-E[\Psi(X)], & \mbox{if}~ \gamma\rightarrow 0^{+}\mbox{ and}~ 0<s<+\infty\\
+\infty, & \mbox{if}~ \gamma\rightarrow 0^{+}\mbox{ and}~ s=+\infty.
\end{array}
\right.
\end{eqnarray}
Moreover, in particular, for $\psi(x)=x$, we have
\begin{eqnarray}
CPE_{\gamma}^{\psi}(X)
=\left\{
\begin{array}{ll}
\frac{1}{2}\left[s^2-E(X^2)\right], & \mbox{if}~ \gamma\rightarrow 0^{+}\mbox{ and}~ 0<s<+\infty\\
CPE_{n}^{\psi(x)=x}(X), &  \mbox{if}~ \gamma=n\in\mathbb{N}~\mbox{and}~s=+\infty\\
CPE^{\psi(x)=x}(X), &  \mbox{if}~ \gamma=1~\mbox{and}~s=+\infty\\
+\infty, & \mbox{if}~ \gamma\rightarrow 0^{+}\mbox{ and}~ s=+\infty,
\end{array}
\right.
\end{eqnarray}
where $CPE_{n}^{\psi(x)=x}(X)$ and $CPE^{\psi(x)=x}(X)$ are the shift-dependent generalized cumulative past entropy of order $n$ (see Eq. (1.4) of \cite{kayal2019shift})  and weighted cumulative past entropy (see Eq. (10) of \cite{misagh2016shift}), respectively.

Next, we consider an example to show that even though the fractional generalized cumulative past entropy of two distributions are same, but the WFGCPEs are not same.
\begin{example}\label{ex2.1}
	Consider two random variables $X_{1}$ and $X_{2}$ with respective CDFs $K_{1}(x)=x-a,~0<a<x<a+1$ and $K_{2}(x)=x-(a+1),~a+1<x<a+2<+\infty.$ Then, the fractional generalized cumulative past entropy of $X_{1}$ and $X_{2}$ can be obtained respectively as
	$$CPE_{\gamma}(X_{1})=\frac{1}{2^{\gamma+1}}=CPE_{\gamma}(X_{2}),~\gamma>0.$$
	That is, the fractional generalized cumulative past entropy of $X_{1}$ and $X_{2}$ are same. Indeed, it is expected since the fractional generalized cumulative past entropy is shift-independent (see Propositon $2.2$ of \cite{di2021fractional}). In order to reach to the goal, let us consider $\psi(x)=x.$ Then,
	 $$CPE_{\gamma}^{\psi(x)=x}(X_1)=\frac{1}{3^{\gamma+1}}+\frac{a}{2^{\gamma+1}}~\mbox{and}~CPE_{\gamma}^{\psi(x)=x}(X_2)=\frac{1}{3^{\gamma+1}}+\frac{a+1}{2^{\gamma+1}},$$
	which show that the WFGCPEs of $X_{1}$ and $X_{2}$ are not same. Here, $CPE_{\gamma}^{\psi}(X_1)< CPE_{\gamma}^{\psi}(X_2)$. Further, let $\psi(x)=x^2.$ Thus, we have
	 $$CPE_{\gamma}^{\psi(x)=x}(X_1)=\frac{1}{4^{\gamma+1}}+\frac{2a}{3^{\gamma+1}}+\frac{a^2}{2^{\gamma+1}}~\mbox{and}~CPE_{\gamma}^{\psi(x)=x}(X_2)=\frac{1}{4^{\gamma+1}}+\frac{2(a+1)}{3^{\gamma+1}}+\frac{(a+1)^2}{2^{\gamma+1}},$$
	which also reveal that the WFGCPEs of $X_{1}$ and $X_{2}$ are different from each other.
\end{example}
From Example \ref{ex2.1}, we notice that when ignoring qualitative characteristic of a given data set, the fractional generalized cumulative past entropy of two distributions are same, as treated from the quantitative point of view. However, when we do not ignore it, they are not same.  In Table $1$, we provide closed form expressions of the WFGCPE of various distributions for two choices of $\psi(x)$. Let $K$ and $\bar{K}=1-K$ be the distribution and survival functions of a symmetrically distributed random variable with bounded support $(0,s)$. \cite{di2021fractional} showed that for this symmetric random variable the fractional generalized cumulative residual entropy and the fractional generalized cumulative entropy are same. However, this property does not hold for the weighted versions of the fractional generalized cumulative residual entropy and fractional generalized cumulative entropy. Indeed,
\begin{eqnarray}
CPE_{\gamma}^{\psi}(X)=\frac{1}{\Gamma(\gamma+1)}\int_{0}^{s}\psi(s-x)\bar{K}(x)[-\ln\bar{K}(x)]^{\gamma}dx,~~\gamma>0.
\end{eqnarray}
Particularly, for a symmetric random variable $X$ with $\psi(x)=x$, we have
\begin{eqnarray}
CPE_{\gamma}^{\psi}(X)&=&\frac{s}{\Gamma(\gamma+1)}\int_{0}^{s}\bar{K}(x)[-\ln\bar{K}(x)]^{\gamma}dx-\frac{1}{\Gamma(\gamma+1)}\int_{0}^{s}x\bar{K}(x)[-\ln\bar{K}(x)]^{\gamma}dx\nonumber\\
&=&s CRE_{\gamma}(X)-CRE_{\gamma}^{\psi(x)=x}(X),~\mbox{say},
\end{eqnarray}
where $CRE_{\gamma}(X)$ and $CRE_{\gamma}^{\psi(x)=x}(X)$ are respectively known as the fractional generalized cumulative residual entropy and weighted fractional generalized cumulative residual entropy. \cite{di2021fractional} showed that the fractional generalized cumulative entropy of a nonnegative random variable is shift-independent. 

\cite{golomb1966information} proposed an information generating (IG) function for a PDF $k$ as 
\begin{eqnarray}
G_{\beta}(k)=\int_{0}^{\infty}k^{\beta}(x)dx,~~\beta>0.
\end{eqnarray}
The derivatives of this IG function with respect to $\beta$ at $\beta=1$ yield statistical information measures for a probability distribution. For example, the first order derivative of $G_{\beta}(k)$ with respect to $\beta$ at $\beta=1$ produces negative Shannon entropy measure. For detailed properties of the Shannon entropy, please refer to \cite{shannon1948note}. Very recently, \cite{kharazmi2021jensen} considered the IG function and discussed some new properties that reveal its connections to some other well-known information measures. The authors have also  shown that the IG measure can be expressed based on different orders of fractional Shannon entropy. \cite{kharazmi2021information} studied IG function and relative IG function measures associated with maximum and minimum ranked set sampling schemes with unequal sizes. Along the similar lines, here we define a weighted cumulative past entropy generating function as 
\begin{eqnarray}
G_{\beta}(K)=\int_{0}^{s}\psi(x)K^{\beta}(x)dx, ~~\beta>0,
\end{eqnarray}
where $\psi(x)$ is a positive valued weight function. Clearly, 
\begin{eqnarray}
\frac{d}{d\beta}G_{\beta}(K)|_{\beta=1}=\int_{0}^{s}\psi(x)K(x)\ln K(x) dx=-CPE_{\gamma=1}^{\psi}(X).
\end{eqnarray}
Indeed, higher order derivatives of $G_{\beta}(K)$ yield higher order weighted cumulative past entropy measues.

In the following proposition, we establish that the WFGCPE is shift-dependent. This makes the proposed weighted fractional measure useful in context-dependent situations.
\begin{proposition}
	Let $Y=aX+b,$ where $a>0$ and $b\ge0.$ Then,
	\begin{eqnarray}\label{eq2.7}
	CPE_{\gamma}^{\psi}(Y)=\frac{a}{\Gamma(\gamma+1)}\int_{0}^{s}\psi(ax+b)K(x)[-\ln K(x)]^{\gamma}dx,~\gamma>0.
	\end{eqnarray}
\end{proposition}
\begin{proof}
	The proof follows from $K_{Y}(x)=K(\frac{x-b}{a})$. Thus, it is omitted.
\end{proof}

\begin{table}[h!]
	\caption{\label{tab1} The closed form expressions of the WFGCPE of different distributions. For the case of Fr\`{e}chet distribution, we assume that $\gamma> 3/c.$ }
	\begin{center}
		\begin{tabular}{c c c c c c c}
			\hline\hline
			Model & Cumulative distribution function & $\psi(x)=x$ & $\psi(x)=x^2$ \\ [1ex]
			\hline\hline
			Power distribution & $K(x)=\left(\displaystyle\frac{x}{b}\right)^{c},~0<x<b,~c>0$,  & $\displaystyle\frac{b^2}{c(1+\frac{2}{c})^{\gamma+1}}$ &$\displaystyle\frac{b^3}{c(1+\frac{3}{c})^{\gamma+1}}$\\[3ex]
			\hline
			Fr\`{e}chet distribution & $K(x)=e^{-\displaystyle b x^{-c}},~x>0,~b,~c>0$  & $\displaystyle\frac{b^{\frac{2}{c}}\Gamma(\gamma-\frac{2}{c})}{c\Gamma(\gamma+1)}$ &$\displaystyle\frac{b^{\frac{3}{c}}\Gamma(\gamma-\frac{3}{c})}{c\Gamma(\gamma+1)}$\\[3ex]
			
			
			\hline\hline
		\end{tabular}
	\end{center}
\end{table}

In particular, let us consider $\psi(x)=x.$ Then, after some simplification, form (\ref{eq2.7}) we get
\begin{eqnarray}
CPE_{\gamma}^{\psi}(Y)
=a^2 CPE_{\gamma}^{\psi(x)=x}(X)+ab~ CPE_{\gamma}(X),
\end{eqnarray}
where
\begin{eqnarray}
CPE_{\gamma}^{\psi(x)=x}(X)=\frac{1}{\Gamma(\gamma+1)}\int_{0}^{s}x K(x) [-\ln K(x)]^{\gamma}dx
\end{eqnarray}
and
$
CPE_{\gamma}(X)$ is given by (\ref{eq1.5}). It is always of interest to express various information measures in terms of the expectation of a function of random variable of interest. Define
\begin{eqnarray}
\mu(t)=\int_{0}^{t}\frac{K(x)}{K(t)}dx,
\end{eqnarray}
which is known as the mean inactivity time of $X$. \cite{di2009cumulative} expressed cumulative entropy in terms of the expectation of the mean inactivity time of $X$. Recently, \cite{di2021fractional} showed that the fractional generalized cumulative entropy can be written as the expectation of a decreasing function of $X.$ Below, we get similar findings for the case of the WFGCPE.
\begin{proposition}\label{prop2.2}
Let $X$ be nonnegative absolutely continuous random variable with distribution function $K(.)$ and density function $k(.)$ such that $CPE_{\gamma}^{\psi}(X)<+\infty.$ Then,
\begin{eqnarray}\label{eq2.12}
CPE_{\gamma}^{\psi}(X)=E[\tau_{\gamma}^{\psi}(X)],
\end{eqnarray}
where
\begin{eqnarray}\label{eq2.13}
\tau_{\gamma}^{\psi}(u)=\frac{1}{\Gamma(\gamma+1)}\int_{u}^{s}\psi(x)[-\ln K(x)]^{\gamma}dx,~~\gamma>0.
\end{eqnarray}
\end{proposition}
\begin{proof}
	 Noting $K(x)=\int_{0}^{x}k(u)du$ and applying Fubini's theorem, we have from (\ref{eq2.1}) that
	 \begin{eqnarray*}
	 	CPE_{\gamma}^{\psi}(X)&=&\frac{1}{\Gamma(\gamma+1)}\int_{0}^{s}\psi(x)[-\ln K(x)]^{\gamma}\left(\int_{0}^{x}k(u)du\right) dx\\
	 	&=& \frac{1}{\Gamma(\gamma+1)}\int_{0}^{s}f(u) \left(\int_{u}^{s}\psi(x)[-\ln K(x)]^{\gamma}dx\right)du\\
	 	&=&E[\tau_{\gamma}^{\psi}(X)],
	 \end{eqnarray*}
 where $\tau_{\gamma}^{\psi}(.)$ is given by (\ref{eq2.13}).  This complets the proof.
\end{proof}

Note that (\ref{eq2.12}) reduces to Eq. (20) of  \cite{tahmasebi2020extension}, for $\gamma=n\in\mathbb{N}$ and $s=+\infty.$ For $\psi(x)=1,$ Proposition \ref{prop2.2} turns out as Proposition $2.1$ of \cite{di2021fractional}.

Similar to the normalized cumulative entropy, \cite{di2021fractional} propoosed a normalized fractional generalized cumulative entropy of a random variable $X$ with nonnegative bounded support.  Here, we define a normalized WFGCPE. It is assumed that the weighted cumulative past entropy with a general nonnegative weight function is nonzero and finite, which is given by (see \cite{suhov2015weighted})
\begin{eqnarray}\label{eq2.14}
CPE^{\psi}(X)=-\int_{0}^{s}\psi(x)K(x)\ln K(x)dx,~~\psi(x)\ge0.
\end{eqnarray}
The normalized WFGCPE of $X$ can be defined as
\begin{eqnarray}
NCPE_{\gamma}^{\psi}(X)=\frac{CPE_{\gamma}^{\psi}(X)}{(CPE^{\psi})^{\gamma}}=
\frac{1}{\Gamma(\gamma+1)}\frac{\int_{0}^{s}\psi(x)K(x)[-\ln K(x)]^{\gamma}dx}{\left(\int_{0}^{s}\psi(x)K(x)[-\ln K(x)]dx\right)^{\gamma}},~~\gamma>0.
\end{eqnarray}
Note that
$$\lim_{\gamma\rightarrow 0^{+}}NCPE_{\gamma}^{\psi}(X)=\int_{0}^{s}\psi(x)K(x)dx~~\mbox{and}~~
\lim_{\gamma\rightarrow 1}NCPE_{\gamma}^{\psi}(X)=1.$$

The closed form expressions of the normalized WFGCPE of power and Fr\`{e}chet distributions are presented in Table \ref{tab2} for two choices of the weight functions.
\begin{table}[h!]
	\caption{\label{tab2} The closed form expressions of the normalized WFGCPE of different distributions as Table $1$. For the case of the Fr\`{e}chet distribution, $c> \max\{3/\gamma,3\},~\gamma>0.$ }
	\begin{center}
		\begin{tabular}{c c c c c c c}
			\hline\hline
			Model & $\psi(x)=x$ & $\psi(x)=x^2$ \\ [1ex]
			\hline \hline
			Power distribution &  $\displaystyle\frac{(c+2)^{\gamma-1}}{\Gamma(\gamma+1)b^{2(\gamma-1)}}$  & $\displaystyle\frac{(c+3)^{\gamma-1}}{\Gamma(\gamma+1)b^{3(\gamma-1)}}$\\[3ex]
			\hline
			Fr\`{e}chet distribution &  $\displaystyle\frac{1}{(\Gamma(\gamma+1))^{2}}\frac{c^{\gamma-1}}{b^{\frac{2}{c}(\gamma-1)}}\frac{\Gamma(\gamma-\frac{2}{c})}{(\Gamma(1-\frac{2}{c}))^{\gamma}}$  &$\displaystyle\frac{1}{(\Gamma(\gamma+1))^{2}}\frac{c^{\gamma-1}}{b^{\frac{3}{c}(\gamma-1)}}\frac{\Gamma(\gamma-\frac{3}{c})}{(\Gamma(1-\frac{3}{c}))^{\gamma}}$ \\[3ex]
			
			
			\hline\hline
		\end{tabular}
	\end{center}
\end{table}

\begin{figure}[h]
\begin{center}
\subfigure[]{\label{c1}\includegraphics[height=1.9in]{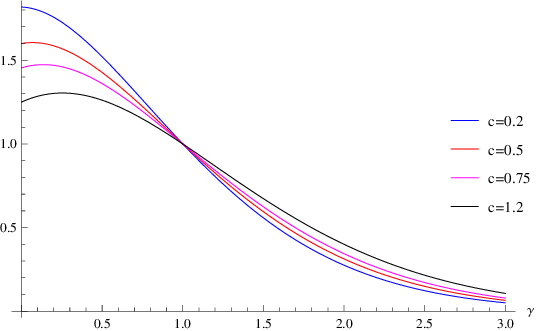}}
\subfigure[]{\label{c1}\includegraphics[height=1.9in]{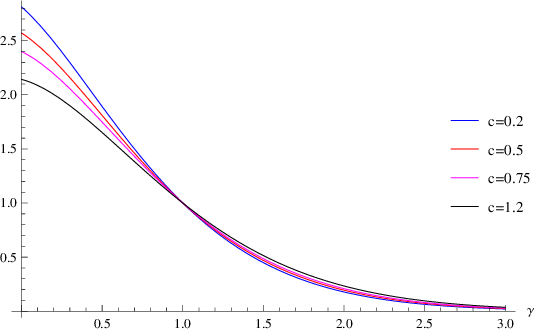}}
\caption{Graphs of the normalized WFGCPE with $(a)$ $\psi(x)=x$ and $(b)$ $\psi(x)=x^2$ for the power distribution as in Table \ref{tab2}.}
\end{center}
\end{figure}

\subsection{Some ordering results}
In this subsection, we obtain some ordering properties for the WFGCPE. It can be shown that the function $\tau_{\gamma}^{\psi}(u)$ given by (\ref{eq2.13}) is decreasing and convex when $\psi(x)$ is decreasing in $x$. Thus, for decreasing $\psi,$ we have
\begin{eqnarray}
X_{1}\le_{dcx}X_{2}\Rightarrow E[\tau_{\gamma}^{\psi}(X_1)]\le E[\tau_{\gamma}^{\psi}(X_2)] \Rightarrow CPE_{\gamma}^{\psi}(X_1)\le CPE_{\gamma}^{\psi}(X_2).
\end{eqnarray}
\cite{di2017further} showed that more dispersed distributions produce larger generalized cumulative entropy. Note that the generalized cumulative entropy was introduced and studied by \cite{kayal2016generalized}. Recently, \cite{di2021fractional} established similar property for the fractional generalized cumulative entropy. In the following proposition, we notice that analogous result holds  for the proposed measure given by (\ref{eq2.1}).  We recall that the dispersive order $X_{1}\le_{disp}X_{2}$ can be equivalently characterized by  (see P. $160$, \cite{oja1981location})
\begin{eqnarray}\label{eq2.15*}
k_{1}(K_{1}^{-1}(u))\ge k_{2}(K_{2}^{-1}(u)),~~u\in(0,1).
\end{eqnarray}
\begin{proposition}\label{prop2.3*}
Consider two nonnegative absolutely continuous random variables $X_{1}$ and $X_{2}$ with CDFs $K_{1}$ and $K_{2}$, respectively. Then,
\begin{eqnarray}\label{eq2.16*}
X_{1}\le_{disp}X_{2}\Rightarrow CPE_{\gamma}^{\psi}(X_{1})\le CPE_{\gamma}^{\psi}(X_{2}),
\end{eqnarray}
provided $\psi$ is increasing.
\end{proposition}
\begin{proof}
Using the transformation $K_{1}(x)=u$, we have
\begin{eqnarray*}
CPE_{\gamma}^{\psi}(X_{1})=\frac{1}{\Gamma(\gamma+1)}\int_{0}^{1}\psi(K_{1}^{-1}(u))u(-\ln u)^{\gamma}\frac{du}{k_{1}(K_{1}^{-1}(u))}.
\end{eqnarray*}
Thus,
\begin{eqnarray}\label{eq2.17}
CPE_{\gamma}^{\psi}(X_{1})-CPE_{\gamma}^{\psi}(X_{2})=\frac{1}{\Gamma(\gamma+1)}\int_{0}^{1}u(-\ln u)^{\gamma}\left[\frac{\psi(K_{1}^{-1}(u))}{k_{1}(K_{1}^{-1}(u))}-\frac{\psi(K_{2}^{-1}(u))}{k_{2}(K_{2}^{-1}(u))}\right]du.
\end{eqnarray}
We know that  $X_{1}\le_{disp}X_{2}\Rightarrow X_{1}\le_{st}X_{2}$, and as a result,  $K_{1}^{-1}(u)\le K_{2}^{-1}(u)$, $u\in(0,1)$. Moreover, $\psi$ is increasing. Thus, $\psi(K_{1}^{-1}(u))\le \psi(K_{2}^{-1}(u)).$ Using this inequality and (\ref{eq2.15*}) in (\ref{eq2.17}), the hypothesis in (\ref{eq2.16*}) holds. This completes the proof.
\end{proof}

Proposition \ref{prop2.3*} reduces to Proposition $1$ of \cite{tahmasebi2020weighted} when $s=+\infty$ and $\gamma=n\in\mathbb{N}.$ In the following, we obtain different sufficient conditions involving the hazard rate order for the similar outcome in (\ref{eq2.16*}). We recall that $X_{1}$ has decreasing failure rate (DFR) if the hazard rate of $X_{1}$ is decreasing, equivalently, $\bar{K}_{1}(x)=1-K_{1}(x)$ is log-convex.
\begin{proposition}
For the random variables $X_{1}$ and $X_{2}$ as in Proposition \ref{prop2.3*}, let $X_{1}\le_{hr}X_{2}$ hold. Further, let $X_{1}$ or $X_{2}$ be DFR. Then, for $\gamma>0,$ one has $CPE_{\gamma}^{\psi}(X_{1})\le CPE_{\gamma}^{\psi}(X_{2})$.
\end{proposition}
\begin{proof}
Making use of the result in Proposition \ref{prop2.3*}, the proof follows from Theorem $2.1(b)$ of  \cite{bagai1986tail}.
\end{proof}

Next, we will study whether the usual stochastic order implies the ordering of the WFGCPE. In doing so, we consider two random variables $X_{1}$ and $X_{2}$ with respective distribution functions $K_{1}(x)=x^{c_1}$ and $K_{2}(x)=x^{c_2}$, $0<x<1$, $c_{1},~c_{2}>0.$ For $c_1\le c_2,$ clearly $K_{1}(x)\ge K_{2}(x)$ implies
$X_{1}\le_{st}X_{2}.$ Now, we plot graphs of the difference of the WFGCPEs of $X_{1}$ and $X_{2}$ in Figure $2$, for some values of $\gamma,$ which reveal that in general, the ordering between the WFGCPEs may not hold.

\begin{figure}[h]
\begin{center}
\subfigure[]{\label{c1}\includegraphics[height=1.5in]{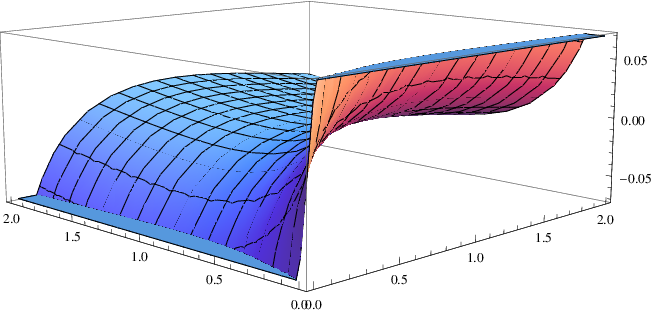}}
\subfigure[]{\label{c1}\includegraphics[height=1.5in]{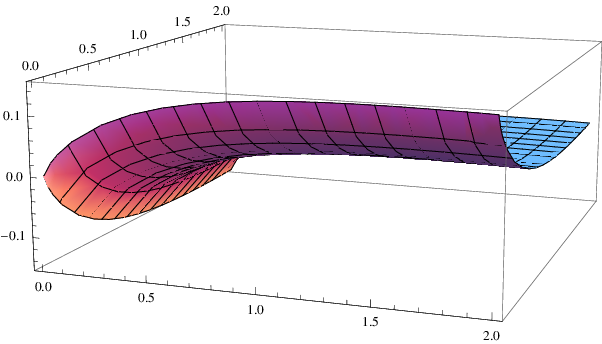}}
\subfigure[]{\label{c1}\includegraphics[height=1.3in]{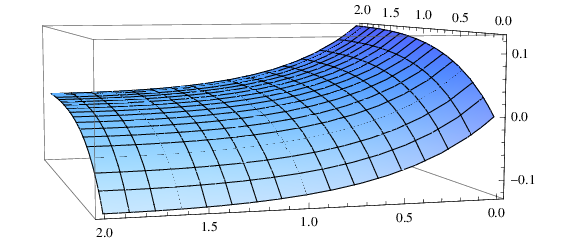}}
\subfigure[]{\label{c1}\includegraphics[height=1.3in]{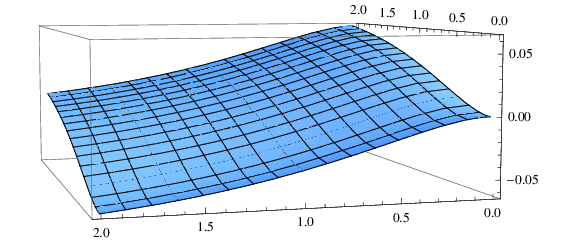}}
\end{center}
\caption{The plots of $CPE_{\gamma}^{\psi}(X_{1})-CPE_{\gamma}^{\psi}(X_{2})$ for $(a)$ $\gamma=0.5$, $(b)$ $\gamma=0.75$, $(c)$ $\gamma=1.5$ and $(d)$ $\gamma=2.5$. Here, $\psi(x)=x$ is considered.}
\end{figure}

We end this subsection with a result which compares the WFGCPE measures when two random variables are ordered in the sense of the usual stochastic order. Here, prime denotes the ordinary derivative.
\begin{proposition}\label{prop2.5}
	Consider two nonnegative absolutely continuous random variables $X_{1}$ and $X_{2}$ with CDFs $K_{1}$ and $K_{2}$, respectively, such that $X_{1}\le_{st}X_{2}$. Further, assume that the means of $X_{1}$ and $X_{2}$ are finite but unequal. Then, for $\tau_{\gamma}^{\psi}(x)<+\infty$ and $E[\tau_{\gamma}^{\psi}(X)]<\infty$,  we have
	\begin{eqnarray}
	 CPE_{\gamma}^{\psi}(X_{1})=E[\tau_{\gamma}^{\psi}(X_{2})]+E[{\tau_{\gamma}^{\psi}}^{\prime}(V)][E(X_{1})-E(X_{2})],
	\end{eqnarray}
	where $V$ is nonnegative absolutely continuous random variable with density function given by
	\begin{eqnarray}
	k_{V}(x)=\frac{\bar{K}_{2}(x)-\bar{K}_{1}(x)}{E(X_{2})-E(X_{1})},~~x>0
	\end{eqnarray}
\end{proposition}
\begin{proof}
	We note that $CPE_{\gamma}^{\psi}(X_1)=E[\tau_{\gamma}^{\psi}(X_1)]$ (see Proposition \ref{prop2.2}). Now, the rest of the proof follows from Theorem $4.1$ of \cite{di1999probabilistic}.
\end{proof}

It can be easily seen that when $\psi(x)=1$, Proposition \ref{prop2.5} reduces to Proposition $3.4$ of \cite{di2021fractional}. Further, when $\gamma=n\in\mathbb{N}$ and $\psi(x)=1$, then Proposition \ref{prop2.5} coincides with Proposition $3.8$ of \cite{di2017further}. Here, ${\tau_{\gamma}^{\psi}}^{\prime}(v)\le 0,$ for $v>0.$ Thus, under the assumptions made in Proposition \ref{prop2.5}, a lower bound of $CPE_{\gamma}^{\psi}(X_{1})$ can be obtained, which is given by
$$CPE_{\gamma}^{\psi}(X_{1})\ge E[\tau_{\gamma}^{\psi}(X_{2})].$$
In the  next subsection, we discuss various bounds of the WFGCPE given by (\ref{eq2.1}).

\subsection{Bounds}
\cite{di2009cumulative} established that the cumulative entropy of the sum of two independent nonnegative random variables is larger than the maximum of the cumulative entropies of the individual  random variables. Similar result was obtained by \cite{di2021fractional} for the fractonal generalized cumulative entropy. Below, we establish analogous result for the WFGCPE under the assumption that the weight function $\psi(x)$ is increasing in $x$ and the PDFs of the random variables are log-concave.
\begin{proposition}\label{prop2.4*}
Let $X_{1}$ and $X_{2}$ be a pair of independent nonnegative absolutely continuous random variables having log-concave PDFs. Then, for all increasing function $\psi$, we have
\begin{eqnarray}
CPE_{\gamma}^{\psi}(X_{1}+X_{2})\ge \max\{CPE_{\gamma}^{\psi}(X_{1}),CPE_{\gamma}^{\psi}(X_{2})\},~~\gamma>0.
\end{eqnarray}
\end{proposition}
\begin{proof}
	Under the assumptions made, the proof follows from  Theorem $3.B.7$ of \cite{shaked2007stochastic} and Proposition \ref{prop2.3*}. Thus, it is omitted.
\end{proof}

When $\gamma=n\in \mathbb{N}$, $s=+\infty$ and $\psi(x)=[K(x)]^{n}$, the result in Proposition \ref{prop2.4*} yields Proposition $2$ of \cite{tahmasebi2020extension}. Further, if we consider $\psi(x)=1$, then one can easily obtain the result stated in Proposition $3.1$ of \cite{di2021fractional}. Next, we obtain a bound of the WFGCPE given by (\ref{eq2.1}).
\begin{proposition}
For a nonnegative random variable with support $(0,s)$ and $\psi(x)=[\xi(x)]^{\gamma}$, $\xi(x)\ge0,$ we have
\begin{eqnarray}
CPE_{\gamma}^{\psi}(X)
\left\{
\begin{array}{ll}
\ge \displaystyle\frac{s^{1-\gamma}}{\Gamma(\gamma+1)}\left[CPE^{\xi}(X)\right]^{\gamma}, & \mbox{if}~ \gamma\ge1\\
\le \displaystyle\frac{s^{1-\gamma}}{\Gamma(\gamma+1)}\left[CPE^{\xi}(X)\right]^{\gamma}, & \mbox{if}~ 0<\gamma\le1,
\end{array}
\right.
\end{eqnarray}
where $CPE^{\xi}(X)=-\int_{0}^{s}\xi(x)K(x)\ln K(x)dx$ is known as the weighted cumulative past entropy with weight function $\xi(x).$
\end{proposition}
\begin{proof}
Let $\gamma\ge1.$ Then, for $0<x<s,$ $K(x)\ge [K(x)]^{\gamma}$. Thus, under the assumptions made, from (\ref{eq2.1}), we obtain
\begin{eqnarray}
CPE_{\gamma}^{\psi}(X)\ge \frac{1}{\Gamma(\gamma+1)}\int_{0}^{s}\alpha_{\gamma}(\beta(x))dx,
\end{eqnarray}
where $\beta(x)=-\xi(x)K(x)\ln K(x)\ge0$ and $\alpha_{\gamma}(t)=t^{\gamma}.$ It can be shown that $t^{\gamma}$ is convex in $t\ge0$, for $\gamma\ge1.$ Thus, from Jensen's integral inequality, the rest of the proof follows. The case for $0<\gamma\le 1$ can be proved similarly.
\end{proof}

\begin{proposition}
	For a nonnegative random variable with support $(0,s)$ and $\gamma>0,$ we have
	\begin{eqnarray}
	CPE_{\gamma}^{\psi}(X)
	\left\{
	\begin{array}{ll}
	\ge \psi(s) CPE_{\gamma}(X), & \mbox{if}~ \psi~ \mbox{is ~decreasing}\\
	\le \psi(s) CPE_{\gamma}(X), & \mbox{if}~ \psi ~\mbox{is ~increasing},
	\end{array}
	\right.
	\end{eqnarray}
	where $CPE_{\gamma}(X)=\frac{1}{\Gamma(\gamma+1)}\int_{0}^{s}K(x)[-\ln K(x)]^{\gamma}dx$ is known as the fractional generalized cumulative past entropy.
\end{proposition}
\begin{proof}
The proof is straightforward, and thus it is omitted.
\end{proof}

\begin{proposition}
	Let $X$ be an absolutely continuous random variable with support $(0,s)$ with mean $E(X)=\mu<+\infty$. Then,
	\begin{itemize}
		\item[(i)] $CPE_{\gamma}^{\psi}(X)\ge \frac{1}{\Gamma(\gamma+1)}\int_{0}^{s}\psi(x)K(x)[1-K(x)]^{\gamma}dx$;
		\item[(ii)] $CPE_{\gamma}^{\psi}(X)\ge D(\gamma)e^{H(X)},$ where $D(\gamma)=e^{\int_{0}^{1}\ln[\psi(K^{-1}(u))u(-\ln u)^{\gamma}]du}$ and $H(X)$ is the differential entropy of $X$;
		\item[(iii)] $CPE_{\gamma}^{\psi}(X)\ge \tau_{\gamma}^{\psi}(\mu),$ provided $\psi$ is decreasing.
	\end{itemize}
	\end{proposition}
\begin{proof} The first part of this proposition follows from the relation $\ln u\le u-1$, for $0<u<1$. To prove the second part,
from the log-sum inequality, we have
\begin{eqnarray}
\int_{0}^{s}k(x)\ln \frac{k(x)}{\psi(x)K(x)[-\ln K(x)]^{\gamma}}dx&\ge& \ln \frac{1}{\int_{0}^{s}\psi(x)K(x)[-\ln K(x)]^{\gamma}dx}\nonumber\\
&=& -CPE_{\gamma}^{\psi}(X).
\end{eqnarray}
Now, the rest of the proof follows using the arguments as in the proof of Theorem $2$ of \cite{xiong2019fractional}.
Third part follows from Jensen's inequality.
\end{proof}

We end this subsection with the following result, which provides bounds of the WFGCPE of $X_{2}$, where the CDF of $X_{2}$ is given by (\ref{eq2.19}).
\begin{proposition}
Let $X_{1}$ and $X_{2}$ be two random variables with CDFs $K_{1}$ and $K_{2}$, respectively. Further, assume that the random variables satisfy proportional reversed hazard model described in (\ref{eq2.19}).  Then,
\begin{eqnarray*}
	CPE_{\gamma}^{\psi}(X_{2})
	\left\{
	\begin{array}{ll}
		\le \eta^{\gamma} CPE_{\gamma}^{\psi}(X_{1}), & \mbox{for}~ \eta\ge1\\
		\ge \eta^{\gamma} CPE_{\gamma}^{\psi}(X_{1}), & \mbox{for}~ 0<\eta\le1.
	\end{array}
	\right.
\end{eqnarray*}
\end{proposition}
\begin{proof}
The proof is simple, and thus it is omitted.
\end{proof}
\subsection{Connection with fractional calculus}
Fractional calculus and its widely application have recently been paid more and more attentions. We refer to \cite{miller1993introduction} and \cite{gorenflo2008fractional} for more recent development on fractional calculus. Several known forms of the fractional integrals have been proposed in the literature. Among these, the Riemann-Liouville fractional integral of order $\gamma> 0$ has been studied extensively for its applications. See, for instance \cite{dahmani2010new}, \cite{romero2013k} and \cite{tuncc2013new}. Let $\gamma>0$ and $f\in L^{1}(a,b),~a\ge0$. Then, the left-sided Riemann-Liouville fractional integral in the interval $[a,b]$ is defined as follows
\begin{eqnarray}\label{eq2.15}
J_{a^{+}}^{\gamma}f(t)=\frac{1}{\Gamma(\gamma)}\int_{a}^{t}\frac{f(\tau)}{(t-\tau)^{1-\gamma}}d\tau,~t\in[a,b],
\end{eqnarray}
where $f$ is a real-valued continuous function. We recall that the notion of the left-sided Riemann-Liouville fractional integral given by (\ref{eq2.15}) can be elongated with respect to a strictly increasing function $h(.)$. In addition to this strictly increasing property, we further assume that the first order derivative $h'(.)$ is continuous in the interval $(a,b)$. Then, for $\gamma>0,$ the left-sided Riemann-Liouville fractional integral of $f$ with respect to $h$ is given by
\begin{eqnarray}\label{eq2.16}
J_{a^{+};h}^{\gamma}f(t)=\frac{1}{\Gamma(\gamma)}\int_{a}^{t}\frac{h'(\tau)f(\tau)}{(h(t)-h(\tau))^{1-\gamma}}d\tau,~t\in[a,b].
\end{eqnarray}
One may refer to \cite{samko1993fractional} (Section $18.2$) for the representation given in  (\ref{eq2.16}).  Now, we will notice that the WFGCPE can be expressed in terms of the limits of the integral (\ref{eq2.16}) after suitable choices of the functions $f(x)$ and $h(x)$, that is, the fractional nature of the proposed measure is justified. Let us take
$$h(x)=\ln K(x)~~\mbox{and}~~f(x)=\frac{\psi(x)(K(x))^{2}}{k(x)}.$$
Then,
\begin{eqnarray}
\lim_{a\rightarrow 0,~t\rightarrow s}J_{a^{+};h}^{\gamma+1}f(t)&=&\frac{1}{\Gamma(\gamma)}\int_{0}^{s}\psi(x)K(x)[-\ln K(x)]^{\gamma}dx\nonumber\\
&=&CPE_{\gamma}^{\psi}(X).
\end{eqnarray}

\subsection{Proportional reversed hazards model}
Let $X$ be a nonnegative absolutely continuous random variable with distribution function $K$ and density function $k.$ Here, $X$ may be treated as the lifetime of a unit. If $\lambda(t)=\frac{d}{dt}\ln K(t)$ denotes the reversed hazard rate of $X$, then $\lambda(t)dt$ represents the conditional probability the unit stopped working in an infinitesimal interval of width $dt$ preceding $t$, given that the unit failed before $t$. In otherwords, $\lambda(t)dt$ is the probability of failing in the interval $(t-dt,t)$ given that the unit is found failed at time $t$.   Let $X_{1}$ and $X_{2}$ be two random variables with PDFs $k_{1}$ and $k_{2}$, CDFs $K_{1}$ and $K_{2}$ and reversed hazard rate functions $\lambda_{1}$ and $\lambda_{2}$, respectively. It is well known that $X_{1}$ and $X_{2}$ have proportional reversed hazard rate model if
\begin{eqnarray}\label{eq2.18}
\lambda_2(x)=\eta \lambda_{1}(x)=\eta\frac{k_{1}(x)}{K_{1}(x)},
\end{eqnarray}
where $\eta>0$ is known as the proportionality constant. Note that (\ref{eq2.18}) is equivalent to the model
\begin{eqnarray}\label{eq2.19}
K_{2}(x)=[K_{1}(x)]^{\eta},~~x\in \mathbb{R},~~\eta>0,
\end{eqnarray}
where $K_{1}$ is the baseline distribution function (see \cite{gupta1998modeling}, \cite{di2000some} and \cite{gupta2007proportional}). The PDF of $X_2$ is
\begin{eqnarray}\label{eq2.20}
k_{2}(x)=\eta (K_{1}(x))^{\eta-1}k_{1}(x),~~x>0,~\eta>0.
\end{eqnarray}
Next, we evaluate the WFGCPE of $X_{2}$. Making use of (\ref{eq2.20}),
from (\ref{eq2.1}) and (\ref{eq2.19}), we have after some standard calculations that
\begin{eqnarray}\label{eq2.21}
CPE_{\gamma}^{\psi}(X_{2})&=&\frac{1}{\Gamma(\gamma+1)}\int_{0}^{s}\psi(x)(K_{1}(x))^{\eta}[-\ln(K_{1}(x))^{\eta}]^{\gamma}dx\nonumber\\
&=&-\frac{1}{\Gamma(\gamma+1)}\Bigg[\int_{0}^{s}x\psi(x)[-\ln(K_{1}(x))^{\eta}]^{\gamma}\eta (K_{1}(x))^{\eta-1}k_{1}(x)dx\nonumber\\
&~&-\gamma\int_{0}^{s}x\psi(x)[-\ln(K_{1}(x))^{\eta}]^{\gamma-1}\eta (K_{1}(x))^{\eta-1}k_{1}(x)dx\nonumber\\
&~&+\int_{0}^{s}\frac{x\psi'(x)}{\eta \lambda_1(x)}[-\ln(K_{1}(x))^{\eta}]^{\gamma}\eta (K_{1}(x))^{\eta-1}k_{1}(x)dx\Bigg]\nonumber\\
&=&-\frac{1}{\Gamma(\gamma+1)}\Bigg[\int_{0}^{s}x\psi(x)[-\ln K_{2}(x)]^{\gamma}k_{2}(x)dx\nonumber\\
&~&-\gamma\int_{0}^{s}x\psi(x)[-\ln K_{2}(x)]^{\gamma-1}k_{2}(x)dx\nonumber\\
&~&+\int_{0}^{s}\frac{x\psi'(x)}{\eta \lambda_1(x)}[-\ln K_{2}(x)]^{\gamma}k_{2}(x)dx\Bigg].
\end{eqnarray}
Now, denote
\begin{eqnarray}\label{eq2.22}
\mathcal{E}_{2}^{\eta}(\gamma)=\frac{1}{\Gamma(\gamma)}E\left[X_2 \psi(X_2)[-\ln K_{2}(X_2)]^{\gamma-1}\right]
\end{eqnarray}
and
\begin{eqnarray}\label{eq2.23}
\tilde{\mathcal{E}}_{2}^{\eta}(\gamma)=\frac{1}{\Gamma(\gamma)}E\left[\frac{X_2 \psi'(X_2)}{\lambda_1(X_{2})}[-\ln K_{2}(X_2)]^{\gamma-1}\right].
\end{eqnarray}
Thus, using (\ref{eq2.22}) and (\ref{eq2.23}) in (\ref{eq2.21}), the following proposition can be obtained.
\begin{proposition}\label{prop2.3}
Let (\ref{eq2.19}) hold. Then, the WFGCPE of $X_{2}$ can be expressed as
\begin{eqnarray}\label{eq2.24}
CPE_{\gamma}^{\psi}(X_{2})=\mathcal{E}_{2}^{\eta}(\gamma)-\mathcal{E}_{2}^{\eta}(\gamma+1)-\eta^{-1}\tilde{\mathcal{E}}_{2}^{\eta}(\gamma+1),~~\gamma>0,
\end{eqnarray}
provided the associated expectations are finite.
\end{proposition}
We note that when $\psi(x)=1$, (\ref{eq2.24}) reduces to Eq. $(19)$ of \cite{di2021fractional}. An illustration of the result in Proposition \ref{prop2.3} is provided in the following example when $\psi(x)=x.$
\begin{example}
	Let $K_{1}(x)=x,~0<x<1$ be the baseline distribution function. We will find the WFGCPE of a random variable $X_{2}$ with distribution function $K_{2}(x)=[K_{1}(x)]^{c}=x^c,~0<x<1,~c>0$. Under this set up, from (\ref{eq2.22}), for $\psi(x)=x$, we obtain
	\begin{eqnarray}\label{eq2.25}
	\mathcal{E}_{2}^{\eta}(\gamma)=\frac{c^{\gamma}}{(2+c)^{\gamma}}=\tilde{\mathcal{E}}_{2}^{\eta}(\gamma).
	\end{eqnarray}
	Now, using (\ref{eq2.25})  in (\ref{eq2.24}), we get
	\begin{eqnarray*}
		 CPE_{\gamma}^{\psi(x)=x}(X_{2})&=&\frac{c^{\gamma}}{(2+c)^{\gamma}}-\frac{c^{\gamma+1}}{(2+c)^{\gamma+1}}-\frac{c^{\gamma}}{(2+c)^{\gamma+1}}\nonumber\\
		&=&\frac{c^{\gamma}}{(2+c)^{\gamma+1}},
	\end{eqnarray*}
which coincides with the case of the Power distribution as in Table \ref{tab1} for $b=1$.
\end{example}

The WFGCPE of $X_{2}$ can be represented in terms of the WFGCPE with different weight functions as follows
\begin{eqnarray}
CPE_{\gamma}^{\psi}(X_{2})=-CPE_{\gamma}^{\psi_{1}}(X_{2})-\eta CPE_{\gamma}^{\psi_{2}}(X_{2})+\eta \gamma^{-1}CPE_{\gamma+1}^{\psi_{2}}(X_{2}),
\end{eqnarray}
where $\psi$ is increasing, $\psi_{1}(x)=x \psi'(x)$ and $\psi_{2}(x)=x\psi(x)\lambda_{1}(x)$. Next, we show that a recurrence relation can be constructed for the WFGCPE of $X_{2}$. It is shown that the WFGCPE of $X_{2}$ of order $(\gamma+1)$ can be expressed in terms of that of order $\gamma$. From (\ref{eq2.24}),
\begin{eqnarray}\label{eq2.27}
CPE_{\gamma+1}^{\psi}(X_{2})&=&\mathcal{E}_{2}^{\eta}(\gamma+1)-\mathcal{E}_{2}^{\eta}(\gamma+2)-\eta^{-1}\tilde{\mathcal{E}}_{2}^{\eta}(\gamma+2)\nonumber\\
&=& \mathcal{E}_{2}^{\eta}(\gamma)-\mathcal{E}_{2}^{\eta}(\gamma+2)-\eta^{-1}[\tilde{\mathcal{E}}_{2}^{\eta}(\gamma+1)+\tilde{\mathcal{E}}_{2}^{\eta}(\gamma+2)]-CPE_{\gamma}^{\psi}(X_{2}).
\end{eqnarray}
Further, when $\psi(x)=1$, (\ref{eq2.27}) reduces to Eq. (22) of \cite{di2021fractional}. We note that the recurrence relation in (\ref{eq2.27}) can be generalized for any integer $n\ge1,$ which is presented in the following proposition.
\begin{proposition} \label{prop2.4}
Let $n$ be a positive integer. Then, under the model in (\ref{eq2.19}), for $\eta>0$ and $\gamma>0$, we obtain
\begin{eqnarray}
CPE_{\gamma+n}^{\psi}(X_{2})&=&\mathcal{E}_{2}^{\eta}(\gamma+n)-\mathcal{E}_{2}^{\eta}(\gamma+n+1)+(-1)^{n-1}[\mathcal{E}_{2}^{\eta}(\gamma)-\mathcal{E}_{2}^{\eta}(\gamma+1)]\nonumber\\
&~&+\eta^{-1}[(-1)^{n}\tilde{\mathcal{E}}_{2}^{\eta}(\gamma+1)-\tilde{\mathcal{E}}_{2}^{\eta}(\gamma+n+1)]+(-1)^{n}CPE_{\gamma}^{\psi}(X_{2}).
\end{eqnarray}
\end{proposition}
\begin{proof}
	The proof follows using similar arguments as in the proof of Proposition $2.4$ of \cite{di2021fractional}.  Thus, it is omitted.
\end{proof}
We note that for the weight function $\psi(x)=1$, Proposition \ref{prop2.4} coincides with Proposition $2.4$ of \cite{di2021fractional}. In this case, the terms $\tilde{\mathcal{E}}_{2}^{\eta}(\gamma+1)$ and $\tilde{\mathcal{E}}_{2}^{\eta}(\gamma+n+1)$ become zero.

\section{Empirical WFGCPE\setcounter{equation}{0}}
Let $T=(T_{1},\ldots, T_{n})$ be a random sample of size $n$ drawn from a population with CDF $K.$ The order statistics of the sample $T$ are the ordered sample values, denoted by $T_{1:n}\le \ldots \le T_{n:n}.$ Denote the indicator function of the set $A$ by $I_{A},$ where
$$
I_{A}
=\left\{
\begin{array}{ll}
1, & \mbox{if }~A~ \mbox{is~ true}\\
0, & \mbox{otherwise}.
\end{array}
\right.
$$
The empirical CDF on the basis of the random sample $T$ is given by
\begin{eqnarray}\label{eq3.1}
\widetilde{K}_{n}(x)=\frac{1}{n}\sum_{i=1}^{n}I_{\{T_{i}\le x\}}
=\left\{
\begin{array}{ll}
0, & \mbox{if} ~~x<T_{1:n}\\
\frac{l}{n}, & \mbox{if}~~T_{l:n}\le x <T_{l+1:n}\\
1, & \mbox{if} ~~ x\ge T_{n:n},
\end{array}
\right.
\end{eqnarray}
where $l=1,\ldots,n-1.$ Using (\ref{eq3.1}), for $\gamma>0$ and $\psi(x)\ge0$, the WFGCPE given by (\ref{eq2.1}) can be expressed as
\begin{eqnarray}\label{eq3.2}
CPE_{\gamma}^{\psi}(\widehat{K}_{n})&=&\frac{1}{\Gamma(\gamma+1)}\int_{0}^{s}\psi(x)\widehat{K}_{n}(x)[-\ln \widehat{K}_{n}(x)]^{\gamma}dx\nonumber\\
&=& \frac{1}{\Gamma(\gamma+1)}\sum_{l=1}^{n-1}\int_{T_{l:n}}^{T_{l+1:n}}\psi(x)\widehat{K}_{n}(x)[-\ln \widehat{K}_{n}(x)]^{\gamma}dx\nonumber\\
&=&\frac{1}{\Gamma(\gamma+1)}\sum_{l=1}^{n-1}Z_{l}\left(\frac{l}{n}\right)\left(-\ln \frac{l}{n}\right)^{\gamma},
\end{eqnarray}
where $Z_{l}=\Psi(T_{l+1:n})-\Psi(T_{l:n})$ and $\Psi(x)=\int_{0}^{x}\psi(x)dx$. Note that when $\psi(x)=1$, we get  the empirical fractional generalized cumulative entropy (see \cite{di2021fractional}) from (\ref{eq3.2}).
For $\psi(x)=x$ and $\gamma=1$, (\ref{eq3.2}) coincides with the empirical weighted cumulative entropy proposed by \cite{misagh2011weighted}. Further, let $\gamma$ be a natural number. Then, for $\psi(x)=x,$ (\ref{eq3.2}) reduces to the empirical shift-dependent generalized cumulative entropy due to \cite{kayal2019shift}. Thus, we observe that the proposed empirical estimate in (\ref{eq3.2}) is a generalization of several empirical estimates developed so far.  In the following theorem, we show that the empirical WFGCPE converges to the WFGCPE almost surely.
\begin{theorem}
	Consider a nonnegative absolutely continuous random variable $X$ with CDF $K.$ Then, for $X\in L^{p},~p>2$, we have
	$$CPE_{\gamma}^{\psi}(\widehat{K}_{n})\rightarrow CPE_{\gamma}^{\psi}(X),$$ almost surely.
\end{theorem}
\begin{proof}
	We have
	\begin{eqnarray}
	\frac{\Gamma(\gamma+1)}{(-1)^{\gamma}}CPE_{\gamma}^{\psi}(\widehat{K}_{n})
	&=&\int_{0}^{1}\psi(x)\widehat{K}_{n}(x)\left[\ln \widehat{K}_{n}(x)\right]^{\gamma}dx+
	\int_{1}^{s}\psi(x)\widehat{K}_{n}(x)\left[\ln \widehat{K}_{n}(x)\right]^{\gamma}dx\nonumber\\
	&=&I_{1}+I_{2},~\mbox{say}.
	\end{eqnarray}
	Now, using dominated convergence theorem and Glivenko-Cantelli theorem, the rest of the proof follows as in Theorem $14$ of \cite{tahmasebi2020extension}.
\end{proof}

Next, we consider a data set, which was studied by \cite{abouammoh1994partial}. It represents the ordered lifetimes (in days) of $43$ blood cancer patients, due to one of the Ministry of
Health Hospitals in Saudi Arabia.
\\
---------------------------------------------------------------------------------------------------------------------\\
$115, 181, 255, 418, 441, 461, 516, 739, 743, 789, 807, 865, 924, 983, 1024, 1062, 1063, 1165,
1191,\\
1222, 1222, 1251, 1277, 1290, 1357, 1369, 1408, 1455, 1478, 1549, 1578, 1578, 15999, 1603,
1605,\\ 1696, 1735, 1799, 1815, 1852, 1899, 1925, 1965.$\\
---------------------------------------------------------------------------------------------------------------------\\
Based on this dats set, tet us now compute the values of the WFGCPE with weight functions $\psi(x)=\sqrt{x}$, $\psi(x)=x$ and $\psi(x)=x^{2}$ for various values of $\gamma,$ which are presented in Table \ref{tab3}. Indeed, one can compute the values of the WFGCPE with any positive valued weight functions.
\begin{table}[h!]
	\caption{\label{tab3} Values of the WFGCPE based on the ordered lifetimes (in days) of $43$ blood cancer patients. }
	\begin{center}
		\begin{tabular}{c c c c c c c c}
			\hline \hline
			$\gamma$ & $\psi(x)=\sqrt{x}$ & $\psi(x)=x$ & $\psi(x)=x^2$ \\ [1ex]
			\hline \hline
			 0.25& 24004.3&  881460    & $1.27542\times10^9$\\[1ex]
			\hline
			 0.5& 20065.8 & 707724  & $9.59358\times10^8$\\[1ex]
\hline
			 0.75& 16858.4 & 570814  & $7.23578\times10^8$ \\[1ex]
\hline
			 1.5& 10279.3 &309581  & $3.22149\times 10^8$\\[1ex]
\hline
			 2.75& 4489.63 &114320  & $8.89639\times10^7$\\[1ex]
			
			
			\hline \hline
		\end{tabular}
	\end{center}
\end{table}

\begin{figure}[h]
\begin{center}
\subfigure[]{\label{c1}\includegraphics[height=1.8in]{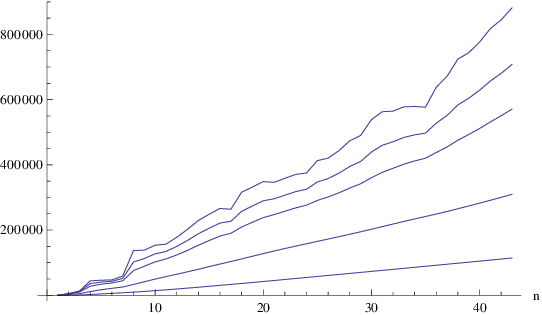}}
\subfigure[]{\label{c1}\includegraphics[height=1.8in]{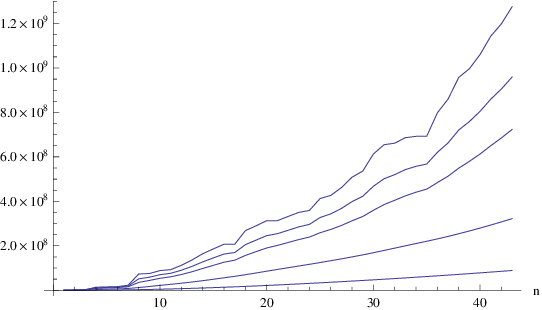}}
\caption{Graphs of the empirical WFGCPE with $(a)$ $\psi(x)=x$ and $(b)$ $\psi(x)=x^2$ based on the ordered lifetimes (in days) of $43$ blood cancer patients, for $\gamma=0.25,0.5,0.75,1.5$ and $2.75$ (from above)}
\end{center}
\end{figure}

Next, we consider examples to illustrate the proposed empirical measure.
\begin{example}\label{ex3.1}
	Let $T=(T_1,\ldots,T_n)$ be a random sample drawn from a population with CDF $K(x)=x^2,~0<x<1.$ Consider $\psi(x)=x.$ It can be shown that $T_{l}^{2}$, $l=1,\ldots,n-1$ follow uniform distribution in the interval $(0,1)$. Further, the sample spacings $T_{l+1:n}^2-T_{l:n}^2$, $l=1,\ldots,n-1$ are independently beta distributed with parameters $1$ and $n$. For details, please refer to \cite{pyke1965spacings}. Thus, from (\ref{eq3.2}), for $\gamma>0,$ we get
	\begin{eqnarray}\label{eq3.3}
	 E[CPE_{\gamma}^{\psi}(\widehat{K}_{n})]=\frac{1}{\Gamma(\gamma+1)}\sum_{l=1}^{n-1}\frac{1}{2(1+n)}\left(\frac{l}{n}\right)\left[-\ln \frac{l}{n}\right]^{\gamma}
	\end{eqnarray}
	and
	\begin{eqnarray}\label{eq3.4}
	 Var[CPE_{\gamma}^{\psi}(\widehat{K}_{n})]=\frac{1}{(\Gamma(\gamma+1))^{2}}\sum_{l=1}^{n-1}\frac{n}{4(1+n)^2(2+n)}\left(\frac{l}{n}\right)^{2}\left[-\ln \frac{l}{n}\right]^{2\gamma}.
	\end{eqnarray}
\end{example}

We present the computed values of the means and variances of the empirical estimator of WFGCPE under the set up explained in Example \ref{ex3.1} in  Table \ref{tab4}.  From Table \ref{tab4}, we observe that for fixed sample sizes, as $\gamma$ increases, the mean and variance of the proposed estimator decrease. Further, for fixed $\gamma,$ the mean and variance respectively increase and decrease, as the sample size increases.

\begin{table}[ht]
	\centering 
	\caption{\label{tab4} Numerical values of
		$E(CPE^{\psi}_{\gamma}(\widehat{\bar{K}}_{n}))$ and  $Var(CPE^{\psi}_{\gamma}(\widehat{\bar{K}}_{n}))$ for the distribution as in Example $3.1.$}
	\scalebox{.95}{\begin{tabular}{c c c c c c c c c c} 
			\hline\hline 
			$\gamma$ & $n$ & $E(CPE^{\psi}_{\gamma}(\widehat{\bar{K}}_{n}))$ &
			$Var(CPE^{\psi}_{\gamma}(\widehat{\bar{K}}_{n}))$  & $\gamma$ & $n$ & $E(CPE^{\psi}_{\gamma}(\widehat{\bar{K}}_{n}))$ &
			$Var(CPE^{\psi}_{\gamma}(\widehat{\bar{K}}_{n}))$\\ [0.5ex] 
			\hline\hline 
			0.25 & 5 & 0.153878 & 0.004609 & 0.5 & 5 & 0.135721 & 0.003395 \\
			~ & 10 & 0.181591 & 0.003434 & ~ & 10 & 0.156472 & 0.002416 \\
			~ & 15 & 0.191238 & 0.002627 & ~ & 15 & 0.163420 & 0.001822 \\
			~ & 30 & 0.200941 & 0.001507 & ~ & 30 & 0.170268 & 0.001034 \\
			~ & 50 & 0.204774 & 0.000956 & ~ & 50 & 0.172941 & 0.000653 \\
			0.75 & 5 & 0.116302 & 0.002472 & 1.5 & 5 & 0.066611 & 0.000968 \\
			~ & 10 & 0.132732 & 0.001734& ~ & 10 & 0.077849 & 0.000712\\
			~ & 15 & 0.138160 & 0.001304 & ~ & 15 & 0.081549 & 0.000538 \\
			~ & 30 & 0.143500 & 0.000738 & ~ & 30 & 0.085119 & 0.000306 \\
			~ & 50 & 0.145593 & 0.000466 & ~ & 50 & 0.086481 & 0.000194 \\
			[1ex] 
			\hline\hline 
	\end{tabular}}
	\label{tb1} 
\end{table}

\begin{example}
	Consider a random sample $T$ from a Weibull population with CDF $K(x)=1-e^{-\theta x^2},~x>0,~\theta>0.$ Using simple transformation theory, it can be established that $T_{i}^{2},~i=1,\ldots,n$ follow exponential distribution with mean $1/\theta$. Further, let $\psi(x)=x.$ Under the present set up, the sample spacings  $T_{l+1:n}^2-T_{l:n}^2$, $l=1,\ldots,n-1$ are independent and exponentialy distributed with mean $1/(\theta(n-l))$ (see  \cite{pyke1965spacings}). Thus, from (\ref{eq3.2}), we obtain
	\begin{eqnarray}
	 E[CPE_{\gamma}^{\psi}(\widehat{K}_{n})]=\frac{1}{\Gamma(\gamma+1)}\sum_{l=1}^{n-1}\frac{1}{2\theta(n-l)}\left(\frac{l}{n}\right)\left[-\ln \frac{l}{n}\right]^{\gamma}
	\end{eqnarray}
	and
	\begin{eqnarray}
	 Var[CPE_{\gamma}^{\psi}(\widehat{K}_{n})]=\frac{1}{(\Gamma(\gamma+1))^{2}}\sum_{l=1}^{n-1}\frac{1}{4\theta^{2}(n-l)^{2}}\left(\frac{l}{n}\right)^{2}\left[-\ln \frac{l}{n}\right]^{2\gamma}.
	\end{eqnarray}
\end{example}

\begin{example}
	Let $T$ be a random sample from a population with absolutely continuous CDF $K$ and PDF $k$. Let $\psi(x)=k(x).$ Then, $Z_{l}=K(T_{l+1:n})-K(T_{l:n}),~l=1,\ldots,n-1$ are independent and beta distributed random variables with parameters $1$ and $n$. We refer to \cite{pyke1965spacings} for details. Thus, similar to (\ref{eq3.3}) and (\ref{eq3.4}),  we have
	\begin{eqnarray}\label{eq3.7}
	 E[CPE_{\gamma}^{\psi}(\widehat{K}_{n})]=\frac{1}{\Gamma(\gamma+1)}\sum_{l=1}^{n-1}\frac{1}{(1+n)}\left(\frac{l}{n}\right)\left[-\ln \frac{l}{n}\right]^{\gamma}
	\end{eqnarray}
	and
	\begin{eqnarray}\label{eq3.8}
	 Var[CPE_{\gamma}^{\psi}(\widehat{K}_{n})]=\frac{1}{(\Gamma(\gamma+1))^{2}}\sum_{l=1}^{n-1}\frac{n}{(1+n)^2(2+n)}\left(\frac{l}{n}\right)^{2}\left[-\ln \frac{l}{n}\right]^{2\gamma}.
	\end{eqnarray}
\end{example}
Hereafter, we provide central limit theorems for the empirical WFGCPE when the random samples are drawn from $(i)$ a Weibull distribution with  $\psi(x)=x$ and $(ii)$ a general CDF $K(x)$ with $\psi(x)=k(x)=\frac{d}{dx}K(x).$

\begin{theorem}
	Consider a random sample $T$ from a population with PDF $k(x)=2\lambda x e^{-\lambda x^2},~x>0,~\lambda>0.$ Then, for $\gamma>0$ and $\psi(x)=x,$
	 $$\frac{CPE_{\gamma}^{\psi}(\widehat{K}_{n})-E(CPE_{\gamma}^{\psi}(\widehat{K}_{n}))}{\sqrt{Var(CPE_{\gamma}^{\psi}(\widehat{K}_{n}))}}\rightarrow N(0,1)$$  in distribution as $n\rightarrow\infty.$
\end{theorem}
\begin{proof}
The proof is similar to that of Theorem $5.1$ of \cite{kayal2019shift}. Thus, it is omitted.
\end{proof}

\begin{theorem}
	Consider a random sample $T$ from a population with CDF $K(x).$ Then, for $\gamma>0$ and $\psi(x)=k(x),$
	 $$\frac{CPE_{\gamma}^{\psi}(\widehat{K}_{n})-E(CPE_{\gamma}^{\psi}(\widehat{K}_{n}))}{\sqrt{Var(CPE_{\gamma}^{\psi}(\widehat{K}_{n}))}}\rightarrow N(0,1)$$  in distribution as $n\rightarrow\infty.$
\end{theorem}
\begin{proof}
	The proof is similar to that of Theorem $15$ of \cite{tahmasebi2020extension}. Thus, it is omitted.
\end{proof}
\section{Concluding remarks and some discussions}
In this paper, we have proposed a weighted fractional generalized cumulative past entropy of a nonnegative random variable having bounded support. A number of results for the proposed weighted fractional measure have been obtained when the weight is a general nonnegative function. It is noticed that WFGCPE is shift-dependent and can be written as the expectation of a decreasing function of the random variable. Some ordering results and bounds are established. Based on the properties, it can be seen that the proposed measure is a variability measure. Further, a connection between the proposed weighted fractional measure and the fractional calculus is provded. The weighted fractional generalized cumulative past entropy measure is studied for the proportional reversed hazards model. A nonparametric estimator of the weighted fractional generalized cumulative past entropy is introduced based on the empirical cumulative distribution function. Few examples are considered to compute mean and variance of the estimator. Finally, a large sample property of the estimator is studied.

The proposed measure is not appropriate when uncertainty is associated with past. Suppose a system has started working at time $t=0$. At a pre-specified inspection time say $t\in(0,s)$, the system is found to be down. Then, the random variable $X_{(t)}=X|X\le t$, where $t\in(0,s)$ is known as the past lifetime. The dynamic weighted fractional generalized cumulative past entropy of  $X_{(t)}$ is defined as
\begin{eqnarray}
CPE_{\gamma}^{\psi}(X;t)=\frac{1}{\Gamma(\gamma+1)}\int_{0}^{t}\psi(x)\frac{K(x)}{K(t)}\left(-\ln \frac{K(x)}{K(t)}\right)^{\gamma}dx,~~\gamma>0,~\psi(x)\ge0.
\end{eqnarray}
One can prove most of the similar properties for $CPE_{\gamma}^{\psi}(X;t)$ as established for the proposed measure given by (\ref{eq2.1}).
\\
\\
{\bf  Acknowledgements:} The author Suchandan Kayal acknowledges the partial financial support for this work under a grant MTR/2018/000350, SERB, India.
\\
\\
{\bf  Conflict of interest statement:} The authors declare that they do not have any conflict of interests.
\\
\\
{\bf  Data availability statement:} NA

\bibliography{sk}
\end{document}